\documentclass[12pt, reqno]{amsart}
\usepackage{amsmath, amsthm, amscd, amsfonts, amssymb, graphicx, color}
\usepackage[bookmarksnumbered, colorlinks, plainpages]{hyperref}
\hypersetup{colorlinks=true,linkcolor=red, anchorcolor=green, citecolor=cyan, urlcolor=red, filecolor=magenta, pdftoolbar=true}
\newtheorem{theorem}{Theorem}[section]
\newtheorem{lemma}[theorem]{Lemma}
\newtheorem{proposition}[theorem]{Proposition}
\newtheorem{corollary}[theorem]{Corollary}
\theoremstyle{definition}
\newtheorem{definition}[theorem]{Definition}
\newtheorem{example}[theorem]{Example}

\theoremstyle{remark}
\newtheorem{remark}[theorem]{Remark}

\usepackage{mathtools}

\newcommand{\ud}[1]{{#1}^{\diamond}}
\newcommand{\udd}[1]{{#1}^{\diamond\diamond}}

\newcommand{\linf}{\mathrel{\text{\raisebox{0.1ex}{\scalebox{0.8}{$\wedge$}}}}}

\newcommand{\norm}[1]{\left\lVert #1 \right\rVert}
\newcommand{\abs}[1]{\left\lvert #1 \right\rvert}

\newcommand{\unto}{\xrightarrow{un}}
\newcommand{\uoto}{\xrightarrow{uo}}
\newcommand{\oto}{\xrightarrow{o}}
\newcommand{\nto}{\xrightarrow{\norm{\cdot}}}

\newcommand{\func}[3]{#1:#2 \rightarrow #3}
\newcommand{\RR}{\mathbb{R}}
\begin{document}

\setcounter{page}{1}

\title{Weak Unbounded Norm Topology and  Dounford-Pettis Operators}
\author[M. Matin]{Mina Matin}
\author[K. Haghnejad Azar]{Kazem Haghnejad Azar}
\author[R. Alavizadeh]{Razi Alavizadeh}
\address{Department  of  Mathematics  and  Applications, Faculty of  Sciences, University of Mohaghegh Ardabili, Ardabil, Iran.}
\email{\textcolor[rgb]{0.00,0.00,0.84}{minamatin1368@yahoo.com}}
\email{\textcolor[rgb]{0.00,0.00,0.84}{haghnejad@uma.ac.ir}}
\email{\textcolor[rgb]{0.00,0.00,0.84}{ralavizadeh@uma.ac.ir}}
\keywords{ $un$-dual, $wun$-topology, $wun$-Dounford-Pettis.
\newline \indent $^{*}$Corresponding author}
\begin{abstract}
	In this paper, we study $un$-dual (in symbol, $\ud{E}$) of Banach lattice $E$ and compare it with topological dual $E^*$. If $E^*$ has order continuous norm, then $E^* = \ud{E}$. We introduce and study weakly unbounded norm topology ($wun$-topology) on Banach lattices and compare it with weak topology and $uaw$-topology. In the final, we introduce and study $wun$-Dunford-Pettis opertors from a Banach lattice $E$ into Banach space $X$ and we investigate  some of its properties and its relationships with others known operators.
\end{abstract}
\maketitle

\section{Introduction}
In \cite{5}, authors shows that $uo$-convergence need not be given by a topology, but $un$-convergence is topological. We will refer to this topology as $un$-topology. The smallest topology $\tau$ that each $un$-continuous functional $f:E\rightarrow \mathbb{R}$ is continuous with respect to that topology is called weakly unbounded norm topology (for short, $wun$-topology), and we denote it by $\tau_{wun}$.
First  we will ours motivate  to write this article.
\begin{enumerate}
	\item We have defined $un$-continuous operators between two Banach lattices $E$ and $F$ in \cite{12o}, and so we introduced the $un$-dual space for a Banach lattice $E$ and we study some of its properties.
	\item Such as the definition of weak topology for a normed space,  we  define  $wun$-topology  for  a Banach lattice and compare it  with  $uaw$-topology  which is introduced in \cite{15}. We show that in general $un$-topology and $wun$-topology are different and by some conditions both topologies coincide. 
	\item By studying of $uaw$-Dounford-Pettis operators in \cite{zabeti1},  it is  interested to define a new generation of operators as $wun$-Dounford-Pettis operators. We study some of its properties and compare  with other known classifications of operators. 
\end{enumerate}
Here we bring some definitions of need.\\
Let $E$ be a vector lattice and $x\in E$.
A net $(x_{\alpha})_{\alpha \in A}\subseteq E$ is said to be order convergent to $x$ if there is a net 
	$(z_{\beta})_{\beta \in B} $ in $ E $ such that
	$ z_{\beta} \downarrow 0 $ and for every $ \beta \in B$,
	there exists $\alpha_{0} \in A$ such that
	$ | x_{\alpha} - x |\leq z_{\beta}$ whenever $ \alpha \geq \alpha_{0}$.
	We denote this convergence by $ x_{\alpha} \xrightarrow{o} x $ and write 
	that $ (x_{\alpha})_{\alpha} $ is $o$-convergent to $x$. In vector lattice $E$ we write 
$x_\alpha\uoto x$ and say that ($x_\alpha$) is $uo$-convergent to $x$ if $|x_\alpha - x|\linf u\oto 0$
for every $u \in E^+$.
In Banach lattice $E$ we write 
$x_\alpha\unto x$ and say that ($x_\alpha$) is $un$-convergent  to $x$ if $|x_\alpha - x|\linf u\nto 0$ 
for every $u \in E^+$.
 
It  was  observed  in  \cite{5} that  $un$-convergence  is 
topological. Let $x_0\in E$ be arbitrary. For every $\epsilon > 0$ and non-zero $u\in X^+$, put
$$V_{\epsilon,u} =\left\{ x \in X : \norm{\abs{x-x_0}\linf u} < \epsilon \right\}.$$
The collection of all sets of this form is a base of $x_0$ neighbourhoods for a topology, and the convergence in 
this topology agrees with $un$-convergence. We will refer to this topology as $un$-topology.\\
 Recall of \cite{15}, a net $(x_\alpha)$ in Banach lattice $E$ is unbounded absolutely weakly convergente ($uaw$-convergent) to $x$ if $(|x_\alpha - x| \wedge u)$ converges to zero weakly for every $u\in E^+$; we write $x_\alpha \xrightarrow{uaw}x$. It  was  observed  in  \cite{15} that  $uaw$-convergence  is  topological. By Theorem 4 of \cite{15}, if Banach lattice $E$ has order continuous norm, then  $(x_\alpha)\subseteq E$ is $un$-null iff it is $uaw$-null in $E$.\\
Let $E$ be a vector lattice and $e\in E^+$. $e$ is weak unit, if band $B_e$ generated by $e$ is equal  with $E$; equivalently, $x \wedge ne \uparrow x$ for every $x\in E^+$, and $e$ is strong unite when ideal $I_a$ generated by $e$ is equal $E$; equivalently, for every $x\geq 0$ there exists $n\in \mathbb{N}$ such that $x\leq ne$. A positive non-zero vector $a$ in a vector lattice $E$ is an atom if the ideal $I_a$ generated by $a$ coincides with span $a$.
We say that $E$ is non-atomic if it has no atoms. We say that $E$ is atomic if $E$ is the band generated by all the atoms.\\
Let $E$ be a vector lattice, $E^\sim$  be the space of all order continuous functionals on $E$ and $(E^{\sim})^\sim$ be the order bidual of $E$. 
Recall that a subset $A$ of $E$ is $b$-order bounded in $E$ 
if $A$ is order bounded in $(E^{\sim})^\sim$. If each $b$-order bounded subset $A$ of vector lattice $E$ is order bounded, we say that $E$ has property $(b)$.\\
Let  $X$, $Y$ be two Banach spaces, then  the continuous operator $T:X\rightarrow Y$ is said to be: 
\begin{itemize}
	\item 
	 \textbf{Dounford-Pettis}
	whenever $(x_n)\subseteq X$ and $x_n \xrightarrow{w}0$, then $T(x_n)\xrightarrow{\|.\|}0$.
	\item 
	  \textbf{weakly compact}
	whenever $T$ carries the closed unit ball of $X$ to a relatively weakly compact subset of $Y$.\\
\end{itemize}

Let  $E$ be a Banach lattice and  $X$  Banach space, then  the continuous operator $T:E\rightarrow X$ is said to be:  
\begin{itemize}	
	\item
	\textbf{$M$-weakly compact}
 if $ \lim \|Tx_{n}\| = 0$ holds for
every norm bounded disjoint sequence $(x_{n}) $ of $ E $. 
\item 
 \textbf{order weakly compact}
whenever $T[0,x]$ is relatively weakly compact subset of $X$ for each $x\in E^+$.
\item  
\textbf{b-weakly compact}
whenever for each $b$-order bounded subset $A$ of $E$, $T(A)$ is a relatively weakly compact. The class of $b$-weakly compact operators was firstly  introduced by  Alpay, Altin and Tonyali \cite{Alpay1}, the class  of  all  $b$-weakly compact operators between $E$ and $X$ will be denoted by
 $W_b(E,X)$.   One of the interesting properties of  the class of $b$-weakly compact operators is that it satisfies the domination property. Some more investigations on $W_b(E,X)$ were done by   \cite{Alpay2, 2c, Hag}. 
 \item 
\textbf{$uaw$-Dunford-Pettis} if for every norm bounded  sequence $(x_{n})$ in $E$,  $ x_{n}\xrightarrow{uaw} 0$ in $ E $ implies
$\| Tx_{n}\|\rightarrow 0$ in $X$. These operators are introduced and examined in \cite{zabeti1}, the class  of  all  $uaw$-Dounford-Pettis operators on $E$  will be denoted by
$B_{UDP}(E)$. This is continued in \cite{Hui}.\\
Moreover if $F$ is a Banach lattice, a continuous operator $T:E\rightarrow F$ is said to be
\item 
\textbf{$un$-continuous} if for every norm bounded and $un$-null net $(x_\alpha)\subseteq E$, $T(x_\alpha)\xrightarrow{un}0$ in $F$. These operators are introduced and examined in \cite{12o}.
\end{itemize}
Recall that a Banach lattice $E$ is said to have dual positive Schur property if every $w^*$-null positive sequence in $E^*$ is norm null.\\

Throughout this article $E$ and $X$ will be assumed to be Banach lattice and Banach space, respectively, and  $(e_n)$ is the sequence of real numbers whose $n^{th}$ term is one and the rest are zero, i.e. $e_n := (0,0,...,0,1,0,0,...)$ unless specified otherwise. For a normed space $X$,  $A\subseteq X$ and $B\subseteq X^*$,  $\sigma (A, B)$ is the smallest topology for $A$ such that each $f\in B$ is continuous on $A$ with respect to this topology.

\section{weakly unbounded norm topology}
 Let $E$ be a Banach lattice. A functional $f:E\rightarrow \mathbb{R}$ is $un$-continuous, if $x_\alpha \xrightarrow{un}0$ implies $f(x_\alpha)\rightarrow 0$ for each norm bounded net $(x_\alpha)\subseteq E$.
We denote the vector space of all $un$-continuous functionals on $E$ by $\ud{E}$ and we call it $un$-dual of Banach lattice $E$.\\
It is clear that $\ud{E}$ is a subspace of $E^*$. The functional $f:\ell^1\rightarrow \mathbb{R}$ defined by $f(x_1,x_2,x_3,...)= \sum_{i=1}^{n}x_i$ is continuous but  is not $un$-continuous. Therefore $\ud{E} \neq E^*$.\\
Let $f\in \ud{E}$, we define $\|f\|_{\ud{E}} = \sup\{|f(x)|: x\in E,\  \|x\|_E\leq 1\}$. It is clear that $\ud{E}$ is a normed space. 

\begin{theorem}\label{po}
	Let $E$ be a Banach lattice. Then we have the following assertions.
	\begin{enumerate}
	\item If  $E^*$ has order continuous norm, then $E^* = \ud{E}$. 
		\item $\ud{E}$ is an ideal in $E^*$.
\item If $E$ is $AM$-space, then $\ud{E}$ is $AL$-space.
\end{enumerate}
\end{theorem}
\begin{proof}
		\begin{enumerate}
		\item Proof follows by   Theorem 6.4 of \cite{5}.
			\item Proof has similar argument of Proposition 5.3 of \cite{12}
			\item Let $E$ be an $AM$-space. Since $E^*$ is an $AL$-space, so $E^*$ has order continuous norm and therefore by part (1), we have  $\ud{E}= E^*$ and it is an $AL$-space. 
\end{enumerate}
\end{proof}

\begin{example}\label{Elin} 
Let $c$ be the sublattice of $\ell^\infty$ consisting of all convergent sequences.	
Since $c^*=\ell^1$ has order continuous norm, therefore by preceding theorem, we have  $\ud{c}= c^* =\ell^1$. 
\end{example}

Theorem \ref{po} shows that $\ud{E}$ is a normed sublattice of a Banach lattice of $E$. Thus we  define the second $un$-dual of Banach lattice $E$ which  show by 
 $\udd{{E}}$.
 $\udd{{E}}$ in general is not equal with topological  second dual of $E$, $E^{**}$. Set $E=c$. It is obvious that 
 $\ud{c^{*}}=\udd{{c}}\neq \ell^\infty=c^{**}$
On the other hand $\ud{E}$ is neither norm closed nor order closed, since $E=\ell_2$, then $\ud{\ell_2}=c_{00}$.

\begin{proposition}\label{abc}
	Let $E$ be a Banach lattice and $G$ be a sublattice of $E$ such that one of the following conditions hold.
	\begin{enumerate}
		\item $ G $ is majorizing in $ E $;
		\item $ G $ is norm dense in $ E $;
		\item $ G $ is a projection band in $ E $.
	\end{enumerate}
	If $\ud{E} = E^*$, then $G^* = \ud{G}$.
	\end{proposition}
	\begin{proof}
		We know that $\ud{G} \subseteq G^*$. Now assume that $ f \in G^*$ and $(x_{\alpha})\subseteq G$ is norm bounded and $un$-null in $G$. By Theorem 3.6 of \cite{14a}, there exists $g \in E^*$ such that $f= g $ on $G$. Note that by assumption $g\in \ud{E}$. By Theorem 4.3  \cite{12}, $x_{\alpha}\xrightarrow{un}0$ in $E$.  We obvious  that $(x_\alpha)$ is norm bounded in $E$. Therefore  $f(x_{\alpha}) = g(x_{\alpha})\xrightarrow{\|.\|}0$ in $\mathbb{R}$. It follows that  $f\in \ud{G}$.
	\end{proof}
Naturally, we can define weakly unbounded norm topology ($wun$-topology) as follows.
\begin{definition}
The smallest topology $\tau$ that each $f\in\ud{E}$  is continuous with respect to that topology is called weakly unbounded norm topology (for short, $wun$-topology), and we denote  by $\tau_{wun}$. In the other words,
$$\tau_{wun}:=\bigcap \left\{ \tau : \text{ each } f\in\ud{E} \text{ is } \tau\text{-countinuous}\right\}.$$ 
\end{definition}
For every $\epsilon>0$ and each $f\in\ud{E}$, put
$$V_{\epsilon,f} =\left\{ x \in X : \abs{f(x)} < \epsilon \right\}.$$
It easily follows from Definition 2.4.2 of \cite{13} that
the collection of all $V_{\epsilon,f}$ is a subbasis for $wun$-topology at zero.\\
Let $x , y \in V_{\epsilon,f}$ and $ 0 \leq \lambda \leq 1$. We have $\abs{f(\lambda x + (1-\lambda)y)} = \lambda |f(x)| + (1-\lambda)|f(y)| \leq \lambda \epsilon + (1-\lambda) \epsilon =  \epsilon$. Therefore $ \lambda x + ( 1 - \lambda) y \in V_{\epsilon,f}$. Hence $\tau_{wun}$ is a locally convex in Banach lattice $E$.\\
Let $E$ be a Banach lattice. It is obvious that  for each norm bounded net $(x_\alpha)\subseteq E$,   $x_\alpha \xrightarrow{wun}0$ if and only if for each $f \in \ud{E}$, $f(x_\alpha)\xrightarrow{}0$ in $\mathbb{R}$.\\

It is obvious that every $un$-null net is $wun$-null in a Banach lattice, but in general the converse not holds. The following example shows that in general  both topologies $un$ and $wun$-topology are not the same.
On the other hand by Proposition 3.5 of \cite{12}, every norm bounded and disjoint net in order continuous Banach lattice $E$ is $wun$-null. Thus if we set $E=\ell^1$, then  $(e_n)$ is $wun$-null in $\ell^1$.

\begin{example}\label{piu}
 Consider the sequence $(e_n)$ in the sublattice $c$ of $\ell^{\infty}$. By Example \ref{Elin}, $\ud{c}=\ell^1$. For each $f=(x_1,x_2,...,x_n,...)\in\ud{c}$, $f(e_n) = x_n \xrightarrow{\|.\|}0$, therefore $e_n \xrightarrow{wun}0$ in $c$, but $(e_n)$ is not $un$-null in $c$. Consider $u = (1,1,1,...)\in c^+$. We have $\|e_n \wedge u\| = \|e_n\| = 1 \nrightarrow 0$.
\end{example}

The following facts are in $wun$-topology that will be used throughout the paper.

\begin{lemma}\label{tr}
	Let $E$ be a Banach lattice and $(x_\alpha)\subseteq E$, then 
	\begin{enumerate}
		\item $x_\alpha \xrightarrow{wun}x$ iff $(x_\alpha - x)\xrightarrow{wun}0$;
		\item $wun$-limits are unique;
		\item If $x_\alpha \xrightarrow{wun}x$ and $y_\beta \xrightarrow{wun}y$, then $ax_\alpha + b y_\beta \xrightarrow{wun} ax + by$, for any scalars $a,b$;
		\item If $x_\alpha \xrightarrow{wun}x$, then $y_\beta \xrightarrow{wun}x$, for every subnet $(y_\beta)$ of $(x_\alpha)$.
	\end{enumerate}
\end{lemma}
\begin{proof}
	\begin{enumerate}
	\item The proof is clear.
	\item Let $(x_\alpha)\subseteq E$ with $x_\alpha \xrightarrow{wun}x$ and $x_\alpha\xrightarrow{wun}y$. For each $
	f \in \ud{E} \subseteq E^*$ we have $\|f (x-y)\| = \| f (x-x_\alpha + x_\alpha -y)\| \leq \| f (x-x_\alpha)\| + \| f (x_\alpha - y)\| $. Therefore $f(x-y)= 0$. Since $E^*$ separates the point of $E$, so $x=y$.
	\item Let $f \in \ud{E}$. We have $\| f (ax_\alpha + b y_\beta - ax - by)\| \leq |a|\|f (x_\alpha -x )\| + |b| \| f (y_\beta - y)\|  $, and proof follows.
	\item Assume that $ (x_\alpha)\subseteq E$ and $x_\alpha \xrightarrow{wun}x$ in $E$. 
Because for each $f \in \ud{E}$, we have $\| f (x_\alpha)\| \rightarrow 0$. It is clear that each subnet $(f( y_\beta))$ of $(f (x_\alpha))$ is norm-null.
	\end{enumerate}
\end{proof}
\begin{remark} 
Since $wun$-limits are unique, therefore for each $x\in E$, $\{x\}$ is $wun$-closed. By condition 2  of Lemma \ref{tr}, $\tau_{wun}$ is a vector topology on $E$, and $(E,\tau_{wun} )$  is a topological vector space. By Theorem 1.12 of \cite{14a}, $\tau_{wun}$ is a Hausdorff topology. 
\end{remark}
Note that $wun$-topology is different with  weak toplogy (in short $w$-topology).  Consider the sequence $(e_n)$ in $\ell^1$.  Since $(e_n) $ is norm bounded and disjoint in $\ell^1$, then by Example \ref{piu}(1), $e_n \xrightarrow{wun}0$ in $\ell^1$. On the other hand,  $(e_n) $ is not  $w$-null in $\ell^1$. Therefore $\tau_w \neq \tau_{wun}$. 
Since   ${\ud{E}}$ is a subspace of $E^*$,  $w$-topology is weaker then $wun$-topology. Thus every $w$-null net is $wun$-null  for each Banach lattice $E$. \\
\begin{proposition}\label{ert}
 Let $E$ be a Banach lattice. If $E$ has strong unit, then $w$-topology and $wun$-topology coincide.
\end{proposition}
\begin{proof}
 Since $\ud{E}\subseteq E^*$, follows that $wun$-topology is weaker than  $w$-topology in Banach lattice $E$. When $E$ has strong unit then by Theorem 2.3 of \cite{12}, $\ud{E}=E^*$ and therefore,  $\tau_w=\tau_{wun}$ in $E$.
 \end{proof}

O. Zabti in \cite{15} has been introduced unbounded absolutely weakly topology (in shorn $uaw$-topology) and investigated some of its properties. 
Note that $wun$-topology is different with $uaw$-topology.  By Lemma 2 of \cite{15}, the sequence $(e_n)\subseteq \ell^\infty$ is $uaw$-null in $\ell^{\infty}$, but $(e_n)$ is not weak convergent to zero in $\ell^\infty$, and so by Proposition \ref{ert}, it is not $wun$-null in $\ell^\infty$.

\begin{remark}
Note that if Banach lattice $E$ has order continuous norm and $(x_\alpha)\subseteq E$ is norm bounded and $uaw$-null, then by Proposition 5 of \cite{15}, $x_\alpha \xrightarrow{wun}0$ in $E$. If $E$ is an $AM$-space with strong unit and $(x_n)\subseteq E$ with $x_n \xrightarrow{wun}0$, then $x_n\xrightarrow{w}0$ and by Exercise 5 of page 355 of \cite{1}, $|x_n|\xrightarrow{w}0$. It follows that  $x_n \xrightarrow{uaw}0$ in $E$.
\end{remark}
\begin{proposition}\label{samedual}
If a Banach lattice $E$ is atomic with order continuous norm, then $(E,\tau_{wun})^*=\ud{E}$. 
\end{proposition}
\begin{proof}
It follows from Theorem 5.2 of \cite{12} that $un$-topology is locally convex. Therefore, $\ud{E}$ separates the points of $E$. Thus, by Theorem 3.10 of \cite{14a} we have $(E,\tau_{wun})^*=\ud{E}$.
\end{proof}
Let $G$ be a sublattice of $E$ and $(x_\alpha)\subseteq G$. By Theorem 3.6 of \cite{14a}, $x_\alpha \xrightarrow{w}0$ in $G$ iff  $x_\alpha\xrightarrow{w}0$ in $E$. The situation is different for $wun$-convergence.
\begin{example}
	\begin{enumerate}
	\item Consider the sequence $(e_n)$ of $\ell^1$. It is $wun$-null in $\ell^1$ while  is not $wun$-nul in $\ell^{\infty}$.
\item Note that $\ell^1$ is an order continuous  Banach lattice with a weak unit $e$. It is known that $\ell^1$ can be represented as an order and norm dense ideal in $L_1(\mu)$ for some finite measures $\mu$. Consider the sequence $( x_n) = (\dfrac{1}{2}(e_1 - e_n))\subseteq \ell^1$.  Since $L_1(\mu)$ has order continuous norm and  is non-atomic, therefore by Corollary 5.4 of \cite{12}, $x_n \xrightarrow{wun}0$ in $L_1(\mu)$. On the other hand $(x_n)$ is not $wun$-null in $\ell^1$. Since $(x_n)$ is $un$-convergent to $\dfrac{1}{2}e_1$ in $\ell^1$ and therefore is $wun$-convergent to $\dfrac{1}{2}e_1$ in $\ell^1$.
\end{enumerate}
\end{example}
Let $G$ be a sublattice of a Banach lattice $E$.
In the following,  we bring the conditions that if a net $(x_\alpha)\subseteq G$ is $wun$-null in  $G$,  then is $wun$-null in $E$ and vic versa.
\begin{theorem}\label{norm}
	Let  $G$ be a sublattice of Banach lattice $E$ and  $(x_{\alpha})\subseteq G$.
	 \begin{enumerate}
\item	If $E^* = \ud{E}$ and $x_{\alpha}\xrightarrow{wun}0 $ in $E$, then  $x_{\alpha}\xrightarrow{wun}0$ in $G$.
	\item  If $G^* = \ud{G}$ and $x_{\alpha}\xrightarrow{wun} 0$ in $G$, then $x_{\alpha}\xrightarrow{wun}0$ in $E$.
	\end{enumerate}
\end{theorem}
\begin{proof}
	\begin{enumerate}
\item	Let $(x_{\alpha})\subseteq G$ and $ x_{\alpha}\xrightarrow{wun}0$ in $E$. By $E^* = \ud{E}$, $x_{\alpha} \xrightarrow{w}0$ in $E$. By Theorem 3.6 of \cite{14a}, $x_{\alpha} \xrightarrow{w}0$ in $G$ and therefore $x_{\alpha} \xrightarrow{wun}0$ in $G$.
	\item If $x_{\alpha}\xrightarrow{wun}0$ in $G$, then $x_\alpha \xrightarrow{w}0$ in $G$ and therefore $x_\alpha \xrightarrow{w}0$ in $E$. So $x_\alpha \xrightarrow{wun}0$ in $E$.
	\end{enumerate}
\end{proof}
\begin{corollary}
	Let $G$ be a sublattice of Banach lattice $E$ and $(x_\alpha)\subseteq G$. If $E^* = \ud{E}$ and $G$ is a  majorizing or norm dense or
		  band projection in $E$, then $x_\alpha \xrightarrow{wun}0$ in $E$ iff $x_\alpha\xrightarrow{wun}0$ in $G$.
\end{corollary}
Let $\func{T}{E}{F}$ be a continuous operator between two Banach lattices. 
Then $T$ has a $un$-adjoint if there exists the unique operator $\ud{T}:\ud{F}\rightarrow\ud{E}$ satisfying
$$
<\ud{T}\ud{y},x> = <\ud{y},Tx> = \ud{y}(Tx),\qquad \forall \ud{y}\in \ud{F}, \forall x\in E.
$$
It easily follows from $\ud{F}\subseteq F^*$ that $\ud{T}=T^*|_{\ud{F}}$.

\begin{theorem}\label{mf}
	Let $\func{T}{E}{F}$ be an operator between two Banach lattices. If $T$ is $un$-continuous, then
	$T$ has $un$-adjoint.
\end{theorem}
\begin{proof}
	Assume that $T$ is $un$-continuous. It is enough to prove that
	$\ud{T}(\ud{F})\subseteq \ud{E}$. Let $\ud{y}\in\ud{F}$ and $(x_\alpha)$ be norm bounded and $un$-null net in $E$. Since $T$ is $un$-continuous, we have $(Tx_\alpha)$ is norm bounded and $un$-null in $F$. As $\ud{y}$ is $un$-continuous, $\ud{T}\ud{y}(x_\alpha)=\ud{y}(Tx_\alpha)\unto 0$. Thus, $\ud{T}\ud{y}\in\ud{E}$.
\end{proof}

\begin{example}\label{uy}
	The operator $T:C[0,1] \rightarrow c_0$, given by 
	$$ T(f) = (\int_0 ^1 f(x)\sin x dx, \int_0 ^1 f(x)\sin2x dx,...)$$ is a $un$-continuous. By Theorem \ref{mf}, $T$ has $un$-adjoint. We have $\ud{T} : \ud{c_0} \rightarrow \ud{(C[0,1])}$, given by\\
	$$  \langle\ud{T}(x_1,x_2,...),f\rangle = \langle(x_1,x_2,...),Tf\rangle = \sum_{n=1}^\infty x_n \int_0^1 f(x)\ sinnx\ dx,$$
	for all 	$ (x_1,x_2,...)\in \ud{c_0}$  and  $f\in C[0,1].$\\
\end{example}

Now, assume that $Q_E$ be a natural mapping from $E$ into $E^{**}$ where $\langle x^\prime, Q_E(x)\rangle=\langle x, x^\prime\rangle=x^\prime (x)$ for all $x\in E$ and $x^\prime\in E^*$. 
Since $\udd{{E}}$ is a subspace of $E^{**}$,  we have the following lemma.

\begin{lemma}\label{lm1}
Let $E$ be a Banach lattice. Then $Q_E(E)\subseteq \udd{{E}}$. 
\end{lemma}   
\begin{proof}
Let $(x^\prime_\alpha)\subseteq \ud{E}$ be a norm bounded and $un$-null net in $\ud{E}$. By Theorem \ref{po} we know that    $\ud{E}$ is an ideal in $E^*$,  and so  $x^\prime_\alpha \wedge y^\prime \in \ud{E}$ for all $y^\prime\in E^*$. 
It follows that $x^\prime_\alpha \wedge (x^\prime_\alpha \wedge y^\prime )\xrightarrow{\Vert .\Vert}0$ in $\ud{E}$. Thus for each $y^\prime \in E^*$, we have 
$x^\prime_\alpha \wedge y^\prime\xrightarrow{\Vert .\Vert}0$ in $\ud{E}$.\\ 
Let $x\in E$. If $x=0$, $Q_E(0)\in  \udd{{E}}$, and proof holds. Now assume that  $x\neq 0$. Then there exists $y^\prime\in E^*$ such that $y^\prime (x)=1$.
Then we have $(y^\prime)^+(x)\geq1$. 
Since
$x^\prime_\alpha (x) \wedge (y^\prime)^+(x)\xrightarrow{\Vert .\Vert}0$ in $\ud{E}$,  follows that $x^\prime_\alpha(x)\xrightarrow{\Vert .\Vert}0$ in $\ud{E}$, and so $Q_E(x)(x^\prime_\alpha )\rightarrow 0$. It follows that $Q_E(x)\in \udd{{E}}$ and proof follows.
\end{proof}

 Now the Lemma \ref{lm1} make motivation to us for definition  a new  topology  for $\ud{E}$,  that is,    the smallest topology on $\ud{E}$ such that each $Q_E(x)$ is continuous with respect to it where $x\in  E$.  This topology is called weak$^*$ unbounded topology (for short $w^*un$-topology). In this way,  $ x_{\alpha}\xrightarrow{w^*un}0$  if and only if $x^\prime (x)\rightarrow 0$ for all $x\in E$. It is clear that the $w^*un$-topology in $\ud{E}$ is a subtopology of $w^*$-topology in $E^*$,  and $w^*un$-topology  is a subset of $wun$-topology in $\ud{E}$, that is, 
 $\sigma (\ud{E},Q_E(E))\subseteq\sigma (\ud{E}, \udd{{E}})\subseteq  \sigma (\ud{E},(\ud{E})^*)\subseteq \sigma (\ud{E}, E^{**})$.

\begin{theorem}
Let $E$ be a Banach lattice. Then $B_{\ud{E}}=\{x^\prime \in \ud{E}:~\Vert x\Vert\leq 1\}$ is $w^*un$-compact.
\end{theorem}
\begin{proof}
It is obviously that $A\subseteq \ud{E}$ is $w^*un$-closed in $\ud{E}$ if and only if there exists $B\subseteq E^*$ which is $w^*$-closed in $E^*$  and $A=B\cap\ud{E}$.  Since $B_{\ud{E}}=B_{E^*}\cap\ud{E}$ and $B_{E^*}$ is $w^*$-compact, proof follows.
\end{proof}

In the following we have some facts for  $wun$-topology and $w^*un$-topology in $\ud{E}$ which  theirs proofs  has similar arguments such as classical studying for $w^*$ and $w$-topologies in $E^*$. 

\begin{corollary} Suppose that
$E$
and
$F$
are Banach lattices. Then we have the following assertions.
\begin{enumerate}
\item  If
$T\in B(E, F)$,
then
$\ud{T}$
is $w^*un$-$w^*un$ continuous. Conversely, if
$S$
is a $w^*un$-$w^*un$ continuous linear operator from
$\ud{F}$
into
$\ud{E}$,
then there is a
$T$
in
$B(E, F)$
such that
$\ud{T}=S$.
\item If
$T\in B(E, F)$, then 
$\udd{{T}}Q_E(E)\subseteq Q_F(F)$
and 
$Q^{-1}_F\udd{{T}}Q_E =T$.
\item A bounded linear operator from a Banach lattice into a Banach lattice is $wun$-compact if and only if its adjoint is $wun$-compact.
\item Suppose that
$T\in B(E, F)$
and 
$Q_F$
is the natural map from 
$F$
into
$\udd{{F}}$.
Then 
$T$
is $wun$-compact if and only if
$\udd{{T}}(\udd{{E}})\subseteq Q_F(F)$.

\end{enumerate}

\end{corollary}

	\begin{definition}
		Let $E$ and $F$ be two Banach lattices.
		A continuous operator $T: E \rightarrow F $ is said to be, weak unbounded norm continuous (or, $wun$-continuous for short),
		if $ x_{\alpha} \xrightarrow{wun} 0$ in $ E $ implies
		$ Tx_{\alpha} \xrightarrow{wun} 0 $ in $F$ for each norm bounded net $(x_\alpha)\subseteq E$. The collection of all $wun$-continuous operators from $E$ to $F$,  will be denoted by
		$ L_{wun}(E,F) $.
	\end{definition}

	\begin{example}
		\begin{enumerate}
			\item	Let $G$ be a majorizing or norm dense or band projection of $\ell^{\infty}$. Then each continuous operator from $G$ to $\ell^{\infty}$ is $wun$-continuous.
			\item Consider the functional $f :\ell^{\infty} \rightarrow \mathbb{R}$ defined with 
			$$f(x_1,x_2,...) = \lim_{n\rightarrow\infty} x_n.$$
		 Since $f$ is positive, 	$f$ is continuous. Now if $(x_n)\subseteq \ell^\infty$ is norm bounded and $x_n \xrightarrow{wun}0$ then $x_n\xrightarrow{w}0$ and therefore $f(x_n)\xrightarrow{w}0$. Hence $f(x_n)\xrightarrow{wun}0$ in $\mathbb{R}$.
		\end{enumerate}
	\end{example}
\begin{remark}
		Note that the operator  $T:\ell^{1} \rightarrow \ell^{\infty} $ defined by 
		\[
		T(x_{1},x_{2},\ldots) = (\sum_{i=1}^\infty x_i , \sum_{i=1}^\infty x_{i},\ldots)
		\]
		is continuous, while  is not $wun$-continuous. 
\end{remark}	
	
	\begin{theorem}\label{ytr}
A		functional $f:E\rightarrow \mathbb{R}$ is $un$-continuous if and only if is $wun$-continuous.
	\end{theorem}
	\begin{proof}
		Let $f$ is $un$-continuous and $(x_{\alpha})\subseteq E$ is norm bounded with $x_{\alpha}\xrightarrow{wun}0 $ in $E$. Therefore for each $\ud{x} \in \ud{E}$, we have $\ud{x}(x_{\alpha})\xrightarrow{\|.\|}0 $ in $\mathbb{R}$. Since $f$ is $un$-continuous, therefore by Theorem \ref{mf}, $f$ has $un$-adjoint. Hence $\ud{f}(\ud{\mathbb{R}})\subseteq \ud{E}$. Therefore for all $\ud{y} \in \ud{\mathbb{R}}$, we have $\ud{y}(fx_{\alpha}) = \ud{f} \ud{y} (x_{\alpha}) \xrightarrow{\|.\|}0 $ in $\mathbb{R}$. Hence $f(x_{\alpha}) \xrightarrow{wun}0$ in $\mathbb{R}$.\\
		Conversally, let $(x_\alpha)\subseteq E$ is norm bounded and $x_\alpha \xrightarrow{un}0$ in $E$. It is clear that $x_\alpha \xrightarrow{wun}0$ in $E$ and therefore $f(x_\alpha)\xrightarrow{wun}0 $ in $\mathbb{R}$. So $f(x_\alpha)\xrightarrow{\|.\|}0$ in $\mathbb{R}$.
	\end{proof}
\begin{corollary}
	Let $E$ and $F$ be two Banach lattices. Similar to Therem \ref{ytr}, if operator $T:E\rightarrow F$ is $un$-continuous, then  is $wun$-continuous.
\end{corollary}	
\begin{remark} Note that, if Banach lattice $E$ is an atomic $KB$-space, then by Theorem 7.5 of \cite{12}, $B_E$ is $un$-compact. Since $\tau_{wun} \subseteq \tau_{un}$, therefore $B_E$ is $wun$-compact.
	\end{remark}
\begin{theorem}
	Let $E$ be an atomic Banach lattice with order continuous norm. If $A\subseteq E$ is a convex set, then $wun$-closure of $A$ is the same as its $un$-closure, that is; $\overline{A}^{wun}=\overline{A}^{un}$.
\end{theorem}
\begin{proof}
	Since $\tau_{wun}\subseteq\tau_{un}$, $\overline{A}^{un}\subseteq\overline{A}^{wun}$.
	On the other hand, by Theorem 5.2 of \cite{12}, $un$-topology is locally convex, hence if $x\notin\overline{A}^{un}$ then by Theorem 3.13 \cite{2}
	there exists some $f\in\ud{E}$, $\alpha\in\RR$ and an $\epsilon>0$
	such that 
	$$f(a)\leq \alpha < \alpha + \epsilon < f(x),$$
	for all $a\in \overline{A}^{un}$.
	Therefore, $\overline{A}^{un}\subseteq B=\{y: f(y)\leq \alpha\}$. By Proposition \ref{samedual} $f$ is $wun$-continuous, thus $B$ is $wun$-closed. Hence $\overline{A}^{wun}\subseteq B$. Therefore, $x\notin \overline{A}^{wun}$. Consequently, $\overline{A}^{un}=\overline{A}^{wun}$.
\end{proof}
\section{$wun$-Dunford-Pettis operators}
	A continuous operator $T$ from Banach lattice $E$ into Banach space $X$ is a $wun$-Dunford-Pettis whenever $x_n \xrightarrow{wun}0$ in $E$ implies $Tx_n \xrightarrow{\|.\|}0$ in $X$ for each norm bounded sequence $(x_n)\subseteq E$.
\begin{example}\label{we}
Operator $T:C[0,1] \rightarrow \ell^1$, given by 
$$ T(f) = (\dfrac{\int_0 ^1 f(x)\sin x dx}{1^2}, \dfrac{\int_0 ^1 f(x)\sin2x dx}{2^2},...)$$
is a $wun$-Dunford-Pettise operator. Let $(f_n)\subseteq C[0,1]$ is norm bounded and $f_n \xrightarrow{wun}0$. Since $(C[0.1])^* = \ud{(C[0,1])}$, so $f_n \xrightarrow{w}0$ in $C[0,1]$. By continuity of $T$, we have $T(f_n)\xrightarrow{w}0$ in $\ell^1$ and by Schur property of $\ell^1$, $T(f_n)\xrightarrow{\|.\|}0$ in $\ell^1$.
\end{example}
\begin{remark} Let $E$ and $F$ be two Banach lattices and $X$ be a Banach space. If $T:E\rightarrow F$ and $S:F\rightarrow X$ are $wun$-Dounford-Pettis, then $SoT$ is $wun$-Dounford-Pettis.
\end{remark}
It is clear that if $T$ is $wun$-Dunford-Pettis, then it is Dunford-Pettis and $\sigma$-$un$-continuous. The identity operator $I:\ell^1\rightarrow \ell^1$ is a $\sigma$-$un$-continuous, but it is not $wun$-Dounford-Pettis operator.\\
 Here we give an example to illustrate the difference between Dunford-Pettis and $wun$-Dunford-Pettis operators.
\begin{example}\label{poi}
	The operator $T:\ell^{1} \rightarrow \ell^{\infty} $ defined by 
	\[
	T(x_{1},x_{2},\ldots) = (\sum_{i=1}^\infty x_i , \sum_{i=1}^\infty x_{i},\ldots)
	\]
	is a Dunford-Pettis operator
	($\ell^{1}$ has Schur property and $ T $ is a continuous operator). 
	Consider $(e_n)\subseteq \ell^1$. $e_n \xrightarrow{wun}0$ in $\ell^1$. We have $T(e_{n})=(1,1,1,\ldots)$, 
	therefore $ (Te_n) $ is not convergent to zero in norm topology. Thus $ T $ 
	is not $wun$-Dunford-Pettis operator.
\end{example}
\begin{remark}
It is clear that if $E^* = \ud{E}$, then operator $T:E\rightarrow X$ is Dunford-Pettis iff it is $wun$-Dunford-Pettis.
\end{remark}

\begin{proposition}\label{p:4.55}
A linear operator from a Banach lattice into a Banach lattice is $wun$-Dounford-Pettis if and only if it is $wun$-norm\index{$wun$-norm} sequentially continuous.
\end{proposition}
\begin{proof}
The forward implication is an easy consequence of the $wun$-$wun$ continuity of $wun$-Dounford-Pettis operators along with the fact that a subset of a Banach lattice consisting of the terms and limit of a $wun$-convergent sequence is $wun$-compact. The converse follows directly from the fact that $wun$-compact subsets of a normed space are $wun$-sequentially compact.
\end{proof}

\begin{proposition}\label{ky}
	If each  Dunford-Pettis operator $T:E \rightarrow F$ between two Banach lattices is $wun$-Dunford-Pettis, then the norm of $E^*$ is order continuous or $F=\{0\}$.
\end{proposition}
\begin{proof}
The proof  has the similar argument   of Theorem 3.1 of \cite{Hui}.	
\end{proof}
\begin{theorem}
	Let  $F\neq \{0\}$ be a reflexive Banach lattice. The zero operator is the only  $wun$-Dounford-Pettis positive operator $T:\ell^1\rightarrow F$.
\end{theorem}
\begin{proof}
	Let $T:\ell^1\rightarrow F$ be a positive operator. Since $F$ is reflexive, then by Theorem 5.29 of \cite{1}, $T$ is a weakly compact operator. By Theorem 5.85 of \cite{1}, $\ell^1$ has Dounford-Pettis property. Therefore by Theorem 5.82 of \cite{1}, $T$ is Dounford-Pettis. Since the norm of $(\ell^1)^*$ is not order continuous and $F\neq \{0\}$, so by Proposition \ref{ky}, $T$ is not $wun$-Dounford-Pettis.
\end{proof}
\begin{remark}
	It is known that every compact operator between Banach lattices is Dunford-Pettis. In the case of a $wun$-Dunford-Pettis operator, the situation is
	different. The Example \ref{poi} is compact while it is not $wun$-Dounford-Pettis.
\end{remark}
Here we give an example that it illustrate $uaw$-Dunford-Pettis operators differ from  $wun$-Dunford-Pettis operators.
\begin{example}\label{elin}
	For each continuous operator $T:C[0,1] \rightarrow c_0$,  the adjoint operator $T^* : \ell^1 \rightarrow (C[0,1])^*$ is a $uaw$-Dunford-Pettis. Indeed let $(x_n)\subseteq \ell^1$ is norm bounded and $x_n\xrightarrow{uaw}0$. By Proposition 5 of \cite{15}, $x_n \xrightarrow{w^*}0$ in $\ell^1$ and therefore $T^*(x_n)\xrightarrow{w^*}0$ in $(C[0,1])^*$. Since $C[0,1]$ has dual positive Schur property, so we have $T^*(x_n)\xrightarrow{\|.\|}0$ in $(C[0,1])^*$. Note that for each continuous operator $T:C[0,1]\rightarrow c_0$, the adjoint operator $T^*: \ell^1 \rightarrow (C[0,1])^*$ is Dunford-Pettis (we know that $T^*$ is continuous and $\ell^1$ has Schur property). Since $(\ell^l)^* = \ell^\infty $ does not has order continuous norm and $(C[0,1])^* \neq 0$, therefore by Proposition \ref{ky}, there exists some $T: C[0,1]\rightarrow c_0$ such that $T^*$ is not $wun$-Dunford-Pettis. 
\end{example}
\begin{remark}
	If $T:E\rightarrow X$ is a $wun$-Dounford-Pettis where $E$ has order continuous norm, then $T$ is a $M$-weakly compact and therefore by Theorem 1 of \cite{zabeti1}, it is a $uaw$-Dounford-Pettis.
\end{remark}
\begin{theorem}\label{elia}
	Let $T:E\rightarrow X$ be an operator from $AM$-space $E$  to Banach space $X$. Then the following assertions are equivalent:
	\begin{enumerate}
		\item $T$ is $M$-weakly compact.
		\item $T$ is weakly compact.
\item $T$ is Dunford-Pettis.
	\item $T$ is $wun$-Dunford-Pettis.
\item $T$ is $uaw$-Dunford-Pettis.
\item $T$ is $b$-weakly compact.
\item Moreover if $E$ has property $(b)$, $T$ is order weakly compact.
\end{enumerate}
\end{theorem}
\begin{proof}
	\begin{enumerate}
\item[]	$(1)\Leftrightarrow (2)$ By Theorem 5.62 of \cite{1}, the proof is complete.
\item[]	$(2)\Rightarrow (3)$ By Theorems 5.85 and 5.82 of 
\cite{1}, $ T $ is a Dunford-Pettis operator.
\item[] $(3)\Rightarrow(1)$ Since $E$ is an $AM$-space, then by Theorem 4.23 of \cite{1}, $E^*$ is an $AL$-space and therefore $E^*$ has order continous norm. By Theorem 3.7.10 of \cite{14}, $T$ is $M$-weakly compact.
\item[] $(3)\Leftrightarrow (4)$ Since $E^* = \ud{E}$, therefore $x_n\xrightarrow{w}0$ in $E$ iff $x_n \xrightarrow{wun}0$ in $E$. Hence the proof is clear.
\item[]	$(3)\Leftrightarrow (5)$ Follows from Corollary 3.7 of \cite{Hui}.
\item[] $(2)\Rightarrow (6)$ Since each $b$-order bounded set in Banach lattice is norm bounded, hence the proof is clear.
\item[] $(6)\Rightarrow (2)$ Let $B_E$ be a closed unit ball of $E$. Since $E$ has strong unit, so $B_E$ is an order interval. Therefore $(TB_E)$ is a relatively weakly compact subset of $X$.
\item[] $(6)\Leftrightarrow (7)$ By property $(b)$ of $E$, $A\subseteq E$ is order bounded if and only if it is $b$-order bounded, hence the proof is clear.

	\end{enumerate}
\end{proof} 
	Let $S,T:E\rightarrow F$ be two positive operators satisfying $0\leq S\leq T$ with $T$ is $wun$-Dounford-Pettis, it is clear that $S$ is $wun$-Dounford-Pettis.

We  give an example that an operator $T$ is $wun$-Dounfore-Pettis while its adjoint is not $wun$-Dounford-Pettis and vic versa. In the following  with under certain conditions,  an operator  $T$ is $wun$-Dounford-Pettis iff $T^*$ is $wun$-Dounford-Pettis.

\begin{example}\label{as}
	\begin{enumerate}
		\item  Consider the operator $T:C[0,1] \rightarrow c_0$, given by 
$$ T(f) = (\int_0 ^1 f(x)\sin x dx, \int_0 ^1 f(x)\sin2x dx,...).$$
If $ (f_n)\subseteq C[0,1]$ is norm bounded and $f_n \xrightarrow{wun}0$ then $f_n\xrightarrow{w}0$.  We have \  \  $\|Tf_n\| = \sup_{m\geq 1}|\int_0^1 f_n(t) \ sin mt\  dt | \leq \int_0^1 |f_n(t)|dt\rightarrow 0$. Hence $T$ is a $wun$-Dunforde-Pettise. Similar to Example 3 of \cite{zabeti1}, adjoint of $T$, $T^*$  is not $wun$-Dounford-Pettis.
\item The functional $f:\ell^1 \rightarrow \mathbb{R}$ defined by 
$$f(x_1,x_2,...)= \sum_1^n x_i$$
is not $wun$-Dounford-Pettis, but $f^*$ is $wun$-Dounford-Pettis.
\end{enumerate}
\end{example}
\begin{theorem}
	Let $E$ and $F$ be two Banach lattices such that $E$ and $F^*$ have strong unit. Then $T:E \rightarrow F$ is $wun$-Dounford-Pettis iff $T^*$ is $wun$-Dounford-Pettis.
\end{theorem}
\begin{proof}
	\begin{enumerate}
\item	Let $T: E \rightarrow F$ be a $wun$-Dounford-Pettis. Since $E$ has strong unit, then by Theorem \ref{elia}, $T$ is $b$-weakly compact operator. because $E$ has strong unit therefore it ia an $AM$-space. By Theorem 4.23   of \cite{1}, $E^*$ is an $AL$-space. Each $AL$-space is a $KB$-space. Therefore $E^*$ is a $KB$-space. Hence by Theorem 3.1 of \cite{2c}, $T^*$ is $b$-weakly compact. Since $F^*$ has strong unit, hence by Theorem \ref{elia}, $T^* $ is $wun$-Donuford-Pettis.
\item The Proof has similar argument of $(1)$ and by  Theorem 3.5 of \cite{2c}, proof follows.
\end{enumerate}
\end{proof}
\begin{theorem}
	Let $F$ be a Banach lattic. If for each arbitrary Banach lattice $E$, operator $T:E\rightarrow F$ is $wun$-Dounford-Pettis, then
	\begin{enumerate}
		\item $F$ is $KB$-space.
		\item $T$ is $b$-weakly compact.
		\end{enumerate}
\end{theorem}
\begin{proof}
	\begin{enumerate}
		\item 	Let $c_0$ be embeddable in $F$ and $T:c_0 \rightarrow F$ be this emmbeding. Then there exist two positive constants $K$ and $M$ satisfying 
	$$K\|(x_n)\| \leq \|T(x_n)\| \leq M\|(x_n)\| \ \ \ for \ all \ (x_n)\subseteq c_0.$$
	Consider the sequence $(e_n)\subseteq c_0$. $e_n \xrightarrow{wun}0$ and norm bounded in $c_0$ but $\| T(e_n)\| \geq K\|e_n\| = K> 0$ which contradicts with assumption. Therefore $c_0$ is not embeddable in $F$. Hence by Theorem 4.61 of \cite{1}, $F$ is a $KB$-space.
	\item By past part we have $F$ is $KB$-space. Since $c_0$ is not embeddable in $F$, then by Theorem 4.63 of \cite{1}, there exist a $KB$-space $H$, a lattice homomorphism $Q:E\rightarrow H$ and a continuous operator $S:H\rightarrow F$ such that $T = SoQ$. Let $(x_n)$ be a $b$-order bounded disjoint sequence in $E$. It is clear that $(Q(x_n))$ is also $b$-order bounded and disjoint sequence in $H$. By Lemma 2.1 of \cite{2c}, $Q(x_n)\xrightarrow{\|.\|}0$ in $H$. Thererore $T(x_n) = SoQ(x_n)\xrightarrow{\|.\|}0$. So $T$ is $b$-weakly compact.
	\end{enumerate}
\end{proof}
\begin{remark}
	Note that if $F$ is $KB$-space, then every operator $T$ from a Banach lattice $E$ into $F$, in general,  is not $wun$-Dounford-Pettis. By Example \ref{as} there exists adjoint operator $T^*$ from $\ell^1$ into $KB$-space $(C[0,1])^*$ such that it is not $wun$-Dounford-Pettis.
\end{remark}
\begin{remark}
	Let $E$ and $F$ be two Banach lattices. If  an operator $T:E\rightarrow F$ is $b$-weakly compact, in general, $T$ is not $wun$-Dounford-Pettis necessarily. We know that $\ell^1$ is  $KB$-space. Therefore by Theorem 4.61 of \cite{1}, $c_0$ is not embeddable in $\ell^1$. By Proposition 2.2 of \cite{2c}, for any Banach lattice $E$, each operator from $E$ into $\ell^1$ is $b$-weakly compact. On the other hand,  the identity operator $I:\ell^1\rightarrow \ell^1$ is not $wun$-Dounford-Pettis.	
\end{remark}

\end{document}